\documentclass[11pt]{article}

\usepackage{amstext,amsfonts,amsthm,graphicx,amssymb,amscd,epsfig}
\usepackage{amsmath}
\usepackage[ansinew]{inputenc}    %

\newtheorem{definition}{Definition}[section]
\newtheorem{rem}[definition]{Remark}
\newtheorem{prop}[definition]{Proposition}
\newtheorem{lem}[definition]{Lemma}
\newtheorem{coro}[definition]{Corollary}
\newtheorem{teo}[definition]{Theorem}
\newtheorem{ex}[definition]{Example}
\newcommand{\C}{\;\mbox{{\sf I}}\!\!\!C}

\newcommand{\R}{I\!\!R}

\newcommand{\K}{I\!\!K}

\newfont{\bbb}{msbm10 scaled\magstephalf}     
\def\C{\mbox{\bbb C}}
\def\K{\mbox{\bbb K}}

\def\cod{\operatorname{cod}}

\def\R{\mbox{\bbb R}}

\title{\bf On the simplicity of multigerms}

\author{R. Oset Sinha,  M. A. S. Ruas and R. Wik Atique}

\date{}

\begin{document}

\maketitle
\begin{abstract}
We prove several results regarding the simplicity of germs and
multigerms obtained via the operations of augmentation, simultaneous
augmentation and concatenation and generalised concatenation. We
also give some results in the case where one of the branches is a
non stable primitive germ. Using our results we obtain a list which
includes all simple multigerms from $\mathbb C^3$ to $\mathbb C^3$.
\end{abstract}

\renewcommand{\thefootnote}{\fnsymbol{footnote}}
\footnote[0]{2010 Mathematics Subject Classification: 58K40
(primary), 32S05, 32S70 (secondary).} \footnote[0]{Keywords: stable
maps, simple germs, augmentations, concatenations, multigerms.}
\footnote[0]{The first author is partially supported by FAPESP grant
no. 2013/02381-1 and DGCYT and FEDER grant no. MTM2012-33073. The
second and third authors are partially supported by FAPESP grant no.
2008/54222-6. The second author is partially supported by CNPq grant
no. 303774/2008-8.}

\section{Introduction}

In the last few years the study of classifications of singularities
of map-germs $f:(\K^n,S)\rightarrow(\K^p,0)$ where $\K=\C$ or $\R$
has given a step forward (specially when $|S|=r>1$) by substituting
the classical classification methods by operations in order to
obtain multigerms from germs in lower dimensions and fewer branches.
In \cite{robertamond}, Cooper, Mond and Wik-Atique use the operation
of augmentation and define monic and binary concatenations in order
to obtain all $\mathcal A_e$-codimension 1 corank 1 multigerms with
$n\geq p-1$ and $(n,p)$ in Mather's nice dimensions. In
\cite{nosso}, the authors define further operations such as a
simultaneous augmentation and concatenation and the generalised
concatenation (which includes both the monic and binary
concatenations as particular cases) to obtain all $\mathcal
A_e$-codimension 2 corank 1 multigerms with the same dimension
restrictions. However very little is known about the simplicity of
the multigerms obtained via these operations.

A multigerm $f=\{f_1,\ldots,f_r\}:(\K^n,S)\rightarrow(\K^p,0)$ with
$S=\{x_1,\ldots,x_r\}$ is simple if there exists a finite number of
$\mathcal A$-classes (classes under the action of germs of
diffeomorphisms in the source and target) such that for every
unfolding
$F:(\K^n\times\K^d,S\times\{0\})\rightarrow(\K^p\times\K^d,0)$ with
$F(x,\lambda)=(f_{\lambda}(x),\lambda)$ and $f_0=f$ there exists a
sufficiently small neighbourhood $U$ of $S\times\{0\}$ such that for
every $(y_1,\lambda),\ldots,(y_r,\lambda)\in U$ where
$F(y_1,\lambda)=\ldots=F(y_r,\lambda)$ the multigerm
$f_{\lambda}:(\K^n,\{y_1,\ldots,y_r\})\rightarrow(\K^p,f_{\lambda}(y_i))$
lies in one of those finite classes. In \cite{robertamond}, Cooper,
Mond and Wik-Atique proved that all $\mathcal A_e$-codimension 1
multigerms in Mather's nice dimensions are simple. Hobbs and Kirk in
\cite{hobbskirk} and the third author in \cite{roberta} obtain a
list of all simple multigerms from $\R^2$ to $\R^3$. Kolgushkin and
Sadykov in \cite{kolgushkin} and Zhitomirskii in \cite{zhitomirskii}
deal with simple multigerms of curves. In \cite{nishisimple},
Nishimura gives an upper bound on the multiplicity of a simple
multigerm. These are probably the only references related to the
simplicity of multigerms.  (For the case of $\mathcal
A$-classification of simple monogerms, many papers can be cited such
as
\cite{duplessis},\cite{brucegaffney},\cite{mond},\cite{Goryunov},\cite{rieger},\cite{riegerruas},\cite{marartari},\cite{arnold},\cite{houstonkirk}...)

In this paper we assess the problem of knowing when a multigerm
obtained by one of the operations mentioned above is simple. We also
study the case when the multigerm contains a non-stable branch.
Section 2 introduces the notation and the basic definitions. Section
3 deals with augmentations of monogerms. We prove that if the
augmenting function $g$ is not simple then the resulting
augmentation is not simple. Section 4 deals with how simplicity is
affected when you add an extra branch to a simple germ. The first
subsection deals with simultaneous augmentation and concatenation.
We prove that, with certain hypotheses, a simultaneous augmentation
and concatenation is simple if and only if the augmentation comes
from an $\mathcal A_e$-codimension 1 germ. In the second subsection
we study the simplicity of generalised concatenations. The main
result is that a non-monic generalised concatenation of stable germs
$F$ and $g$ where $F$ has zero dimensional analytic stratum is non
simple (Corollary \ref{maingenconc}). The third subsection deals
with germs where one of the branches is non stable. We classify here
all the simple multigerms $h=\{f,g\}$ where $f$ is a non stable germ
and $g$ is a prism on a Morse function or an immersion.

We give clues to which may be the remaining simple multigerms which
are not classified in this paper, namely multigerms $h=\{f,g\}$ with
$f$ and $g$ stable and the dimensions of their analytic strata
between 1 and $p-2$, and the case where $f$ is non stable and $g$ is
a stable singularity more degenerate than a prism on a Morse
function or an immersion. We prove some partial results and show
examples of these cases.

In the last section we use our results to obtain a list which
includes all simple multigerms from $\mathbb C^3$ to $\mathbb C^3$.

\section{Notation}

Let $\mathcal O_{n}^p$ be the vector space of monogerms with $n$
variables and $p$ components. When $p=1$, $\mathcal O_{n}^1=\mathcal
O_n$ is the local ring of germs of functions in $n$-variables and
$\mathcal M_n$ its maximal ideal. The set $\mathcal O_{n}^p$ is a
free $\mathcal O_n$-module of rank $p$. A multigerm is a germ of an
analytic (complex case) or smooth (real case) map
$f=\{f_1,\ldots,f_r\}:(\K^n,S)\rightarrow (\K^p,0)$ where
$S=\{x_1,\ldots,x_r\}\subset \K^n$, $f_i:(\K^n,x_i)\rightarrow
(\K^p,0)$ and $\K=\C$ or $\R$. Let $\mathcal M_n\mathcal O_{n,S}^p$
be the vector space of such map germs. Let $\theta_{\mathbb K^n,S}$
and $\theta_{\mathbb K^p,0}$ be the $\mathcal O_n$-module of germs
at $S$ of vector fields on $\K^n$ and $\mathcal O_p$-module of germs
at 0 of vector fields on $\K^p$ respectively. We denote them by
$\theta_n$ and $\theta_p$. Let $\theta(f)$ be the $\mathcal
O_n$-module of germs $\xi:(\K^n,S)\rightarrow T\K^p$ such that
$\pi_p\circ \xi=f$ where $\pi_p:T\K^p\rightarrow \K^p$ denotes the
tangent bundle over $\K^p$.

Define $tf:\theta_n\rightarrow \theta(f)$ by $tf(\chi)=df\circ\chi$
and $wf:\theta_p\rightarrow \theta(f)$ by $wf(\eta)=\eta\circ f$.
The $\mathcal A_e$-tangent space of a germ $f$ is defined as
$T\mathcal A_e f=tf(\theta_n)+wf(\theta_p)$ and its $\mathcal
A_e$-codimension, denoted by $\mathcal A_e$-cod($f$), is the
$\K$-vector space dimension of $$NA_e(f)=\frac{\theta(f)}{T\mathcal
A_e f}.$$ When we have the $\mathcal A$-tangent space $T\mathcal A
f=tf(\mathcal M_n\cdot\theta_n)+wf(\mathcal M_p\cdot\theta_p)$ in
the denominator of the previous quotient and $\mathcal M_n\theta(f)$
in the numerator, its dimension is called the $\mathcal
A$-codimension. We refer to Wall's survey article \cite{wall} for
general background on the theory of singularities.

\begin{definition}

i) A vector field germ $\eta\in \theta_p$ is called \textit{liftable
over $f$} if there exists $\xi\in\theta_n$ such that
$df\circ\xi=\eta\circ f$ ($tf(\xi)=wf(\eta)$). The set of vector
field germs liftable over $f$ is denoted by $Lift(f)$ and is an
$\mathcal{O}_p$-module.

ii) Let $\widetilde{\tau}(f)=ev_0$($Lift(f)$) be the evaluation at
the origin of elements of $Lift(f)$.
\end{definition}

In general $Lift(f)\subseteq Derlog(V)$ when $V$ is the discriminant
of an analytic $f$ and $Derlog(V)$ represents the $\mathcal
O_p$-module of vector fields tangent to $V$. We have an equality
when $\K=\C$ and $f$ is complex analytic.

The set $\widetilde{\tau}(f)$ is the tangent space to the well
defined manifold in the target containing 0 along which the map $f$
is trivial (i.e. the analytic stratum). Following Mather, $f$ is
stable if and only if all its branches are stable and their analytic
strata have regular intersection (\cite{mather}).

Given $f=\{f_1,\ldots,f_r\}:(\K^n,S)\rightarrow (\K^p,0)$, let
$m_0(f)=\dim_{\mathbb K}\frac{\mathcal O_{n,S}}{f^*(\mathcal M_p)}$
denote the multiplicity of the germ $f$. Note that $$\dim_{\mathbb
K}\frac{\mathcal O_{n,S}}{f^*(\mathcal M_p)}=\sum_{i=1}^r
\dim_{\mathbb K}\frac{\mathcal O_{n,x_i}}{f_i^*(\mathcal M_p)}.$$

From here on we consider only corank 1 germs. We say that
$f=\{f_1,\ldots, f_r\}$ is of type $A_{k_1,\ldots,k_r}$ if $f_i \in
A_{k_i}, \, i=1, \ldots, r.$ For these singularities,
$m_0(f)=k_1+\ldots+k_r+r.$

\section{Simplicity of Augmentations}

\begin{definition}
Let $h:(\mathbb{K}^n,S)\rightarrow (\mathbb K^p,0)$ be a map-germ
with a 1-parameter unfolding $H:(\mathbb K^n\times\mathbb
K,S\times\{0\})\rightarrow (\mathbb K^p\times\mathbb K,0)$ which is
stable as a map-germ, where $H(x,\lambda)=(h_{\lambda}(x),\lambda)$,
such that $h_0=h$. Let $g:(\mathbb K^q,0)\rightarrow (\mathbb K,0)$
be a function-germ. Then, the \textit{augmentation of h by H and g}
is the map $A_{H,g}(h)$ given by $(x,z)\mapsto (h_{g(z)}(x),z)$. A
germ that is not an augmentation is called primitive.
\end{definition}

A natural question arises: given simple germs $h$ and $g$, is
$A_{H,g}(h)$ simple? This is not true in general as can be seen in
the following

\begin{ex}
Consider the simple germ $h(x_1,x_2)=(x_1^3+x_2^4x_1,x_2)$ and
$H(x_1,x_2,\lambda)=(x_1^3+x_2^4x_1+\lambda x_1,x_2,\lambda)$. If we
augment $h$ by the simple function $g(z)=z^4$, we obtain the
non-simple germ $A_{H,g}(h)(x_1,x_2,z)=(x_1^3+(x_2^4+z^4)x_1,x_2,z)$
(\cite{marartari}).
\end{ex}

For monogerms we show that the simplicity of the augmenting function
$g$ is a necessary condition for the simplicity of the augmentation.
In fact, we prove that if two augmentations
$f_1(x,z)=(h_{g_1(z)}(x),z)$ and $f_2(x,z)=(h_{g_2(z)}(x),z)$ are
$\mathcal A$-equivalent, then $g_1$ and $g_2$ are $\mathcal
K$-equivalent. The contact group ${\mathcal K}$ is the set of germs
of diffeomorphisms of $\mathbb{K}^n\times \mathbb{K}^p,0$ which can
be written in the form $H(x,y)=(h(x),H_1(x,y)),$ with $h\in$
Diff$(\mathbb{K}^n,0)$ and $H_1(x,0)=0$ for $x$ near $0$. Two
map-germs $g_1$ and $g_2$ are $\mathcal K$-equivalent if there
exists $H\in \mathcal K$ such that $H(x,g_1(x))=(h(x),g_2(h(x)))$.
We need a previous lemma.

\begin{lem}\label{lemma1}
Let $G_i(z,\epsilon)=g_i(z)+\psi(g_i(z),\epsilon)$ for $i=1,2$ such
that $\psi(0,\epsilon)=\phi(\epsilon)$ is homogeneous of degree $d$.
If $G_1\sim_{\mathcal{K}}G_2$, then $g_1\sim_{\mathcal{K}}g_2$.
\end{lem}
\begin{proof}
For any $G(z,\epsilon)=g(z)+\psi(g(z),\epsilon)$ satisfying the
hypotheses we claim that
$G(z,\epsilon)\sim_{\mathcal{K}}g(z)+\phi(\epsilon)$. In fact, let
$g(z)=w$, then
$G(z,\epsilon)=w+\phi(\epsilon)+w\widetilde{\psi}(w,\epsilon)=w(1+\widetilde{\psi}(w,\epsilon))+\phi(\epsilon)$.
If $\epsilon=(\epsilon_1,\ldots,\epsilon_m)$ and let
$\epsilon_i=(1+\widetilde{\psi}(w,\epsilon))^{\frac{1}{d}}\epsilon_i'$
we obtain that $G\sim_{\mathcal{K}}w+\phi(\epsilon')$.

Since $G_1\sim_{\mathcal{K}}G_2$, we have
$g_1(z)+\phi(\epsilon)\sim_{\mathcal{K}}g_2(z)+\phi(\epsilon)$ and
so their Tjurina algebras, $T_i$, are isomorphic. Since $\phi$ is
homogeneous $T_i=\frac{\mathcal O_{z,\epsilon}}{\langle
\frac{\partial g_i}{\partial z},g_i \rangle+\langle
\frac{\partial\phi}{\partial\epsilon} \rangle}$. We have that
$$\frac{\mathcal O_{z}}{\langle \frac{\partial g_1}{\partial z},g_1
\rangle}\cong\frac{T_1}{\mathcal
M_{\epsilon}T_1}\cong\frac{T_2}{\mathcal
M_{\epsilon}T_2}\cong\frac{\mathcal O_{z}}{\langle \frac{\partial
g_2}{\partial z},g_2 \rangle}$$ and the result follows.
\end{proof}

We remark here that by \cite{riegerruas}, any simple germ with $n>p$
comes from a simple germ with $n=p$ by just adding quadratic terms
in the remaining variables, so for the case $n\geq p$ it is enough
to study the equidimensional case.

\begin{prop}\label{kequiv}
Let $h:(\mathbb{K}^n,0)\rightarrow (\mathbb K^p,0)$ with $n\geq p-1$
be a non stable primitive monogerm which admits a 1-parameter stable
unfolding $H$. Let $g_1$ and $g_2$ be augmenting functions and $f_1$
and $f_2$ the corresponding augmentations. Then
$$f_1\sim_{\mathcal{A}}f_2\Rightarrow g_1\sim_{\mathcal{K}}g_2.$$
\end{prop}
\begin{proof}
First suppose that $p=n$. If $h$ is primitive, by \cite{houston},
$\widetilde\tau(H)=\{0\}$ and so $m_0(h)=m_0(H)\geq n+2$. Since $h$
admits a 1-parameter stable unfolding $m_0(h)\leq n+2$ (by
\cite{mather} stable germs have multiplicity $\leq p+1$). Therefore
$m_0(h)=n+2$. From \cite[Lemma 4.10]{riegerruas2} we know that such
a germ is $\mathcal A$-equivalent to
$(x,y^{n+2}+x_1y+\ldots+x_{n-1}y^{n-1})$ if it is $\mathcal
A_e$-codimension 1 or to
$(x,y^{n+2}+x_1y+\ldots+x_{n-1}^ky^{n-1}+x_{n-1}y^n)$ if it is
$\mathcal A_e$-codimension $k$ with $k\geq 2$.

In the first case, a 1-parameter stable unfolding is
$(x,\lambda,y^{n+2}+x_1y+\ldots+x_{n-1}y^{n-1}+\lambda y^n)$.
Consider two $\mathcal A$-equivalent augmentations
$f_i(x,z,y)=(x,z,y^{n+2}+x_1y+\ldots+x_{n-1}y^{n-1}+g_i(z)y^n)$
$i=1,2$. By \cite[Lemma 4.7]{riegerruas2} we have that
$$G_1(x,z)=(x_1,\ldots,x_{n-1},g_1(z))\sim_{\mathcal K}
G_2(x,z)=(x_1,\ldots,x_{n-1},g_2(z))$$ and so $g_1$ and $g_2$ are
$\mathcal K$-equivalent.

In the second case, a 1-parameter stable unfolding is
$(x,\lambda,y^{n+2}+x_1y+\ldots+x_{n-1}^ky^{n-1}+x_{n-1}y^n+\lambda
y^{n-1})$. Considering two $\mathcal A$-equivalent augmentations
$(x,z,y^{n+2}+x_1y+\ldots+x_{n-1}^ky^{n-1}+x_{n-1}y^n+g_i(z)y^{n-1})$
for $i=1,2$, in the same way as above we have that
$$G_1(x,z)=(x_1,\ldots,x_{n-1}^k+g_1(z),x_{n-1})\sim_{\mathcal K}
G_2(x,z)=(x_1,\ldots,x_{n-1}^k+g_2(z),x_{n-1}).$$ Since $G_i(x,z)$
is $\mathcal K$-equivalent to $(x_1,\ldots,x_{n-1},g_i(z))$ we have
the desired result.

Now suppose that $p=n+1$. As in  the equidimensional case,
$\widetilde\tau(H)=\{0\}$. Therefore $n$ is odd, say $n=2l+1$, and
$m_0(h)=m_0(H)=l+2$. From \cite[Proposition 4.5]{rrwa}, if $l\geq
2$, $h$ is equivalent to either:
$$(x_1,\ldots,x_{2l},y^{l+2}+x_1y+\ldots+x_ly^l,x_{l+1}y+\ldots+x_{2l}y^l+\tilde{h}(x,y))$$
or
$$(x_1,\ldots,x_{2l},y^{l+2}+x_1y+\ldots+x_ly^l,x_{l+1}y+\ldots+x_{2l-1}y^{l-1}+x_{2l}y^{l+1}+y^{l+2}+\tilde{h}(x,y))$$
where in both cases $\tilde{h}\in
\mathcal{M}_n^{l+3}\mathcal{O}_{n}^{n+1}$. Then $f_i$, $i=1,2$, can
be either
$$(x_1,\ldots,x_{2l},z,y^{l+2}+\sum_{j=1}^lx_jy^j,x_{l+1}y+\ldots+x_{2l}y^l+g_i(z)y^{l+1}+\tilde{h}(x,y))$$
or
$$(x_1,\ldots,x_{2l},z,y^{l+2}+\sum_{j=1}^lx_jy^j,x_{l+1}y+\ldots+x_{2l-1}y^{l-1}+g_i(z)y^l+x_{2l}y^{l+1}+y^{l+2}+\tilde{h}(x,y)).$$
Given a corank 1 germ $f_i$ we associate a germ $G_i$ whose
component functions define the set of $l+2$-points appearing in a
stable perturbation of $f_i$. If $f_1$ is ${\mathcal{A}}$-equivalent
to $f_2$ then $G_1$ is $\mathcal K$-equivalent to $G_2$. Following
\cite[Section 3.2]{rrwa}, $G_i:(\mathbb{K}^{3l+2+q},0)\rightarrow
(\mathbb K^{2l+2},0)$ with source coordinates
$(x,z,y,\epsilon_2,\ldots,\epsilon_{l+2})$, and we can show that
$G_i$ is $\mathcal K$-equivalent to
$(x,y,g_i(z)+\psi(g_i(z),\epsilon_2,\ldots,\epsilon_{l+2}))$ in both
cases, where $\psi(0,\epsilon)$ is homogeneous. The result can now
be obtained applying Lemma \ref{lemma1}.

\end{proof}

\begin{ex}
\begin{enumerate}
\item[i)]
The augmentation $f(x,z_1,z_2)=(x^3+(z_1^4+z_2^4)x,z_1,z_2)$ of
$h(x)=x^3$ is not simple since the augmenting function
$g(z_1,z_2)=z_1^4+z_2^4$ is not simple.

\item[iii)]
The converse of the proposition is not true. If we take the
primitive germ $(z^2,z^5)$ and augment it by the simple function
$g(x,y)=x^2+y^4$, we obtain the non-simple germ
$(x,y,z^2,z^5+(x^2+y^4)z)$ (see \cite{houstonkirk}).
\end{enumerate}
\end{ex}

\begin{rem}
We think that Proposition \ref{kequiv} also holds for multigerms.
However, we have only been able to extend the arguments in the proof
for particular examples such as a multigerm consisting only of fold
singularities.
\end{rem}

\section{Simplicity of multigerms}

The classification techniques for multigerms developed recently
consist on combining monogerms to obtain multigerms. In this sense
we are interested in knowing what combinations of simple germs yield
simple multigerms. Subsections 4.1 and 4.3 deal with the simplest
combination of germs, which consists of adding a prism on a Morse
function (when $n\geq p$) or an immersion (when $p=n+1$) to a simple
germ. In 4.1 we study the simultaneous augmentation and
concatenation operation and in 4.3 we combine a primitive
codimension 1 germ with a prism on a Morse function or an immersion.
Subsection 4.2 studies combinations of 2 stable germs, in
particular, those arising from generalised concatenations.

In what follows we discuss the codimension of a multigerm where one
of the branches is a prism on a Morse function or an immersion.

We are considering corank 1 multigerms of type $A_{k_1,\ldots,k_r}$,
for which it is known that their corresponding orbits in the
multijet space are defined by submersions in the stable case and by
ICIS in the finitely determined case
(\cite{Goryunov},\cite{mararmond}).

We note that there is a close relation between the $\mathcal
A$-codimension and the $\mathcal A_e$-codimension. This is due to
Wilson's formula (for the monogerm case see \cite{wall}; see
\cite{hobbskirk} too), which asserts that if the
$\mathcal{A}_e$-codimension is different from 0 and $f$ is
$\mathcal{A}$-simple, then
$$\mathcal{A}_e-\cod(f)=\mathcal{A}-\cod(f)+r(p-n)-p,$$ where $r$ is the number of branches.

Let $f=\{f_1,\ldots,f_r\}:(\K^n,S)\longrightarrow (\K^p,y)$ be a
non-stable multigerm with $\mathcal{A}$-codimension $s$. Let's
assume that $f$ is $k$-determined and $\mathcal{A}$-simple. Suppose
there exists a smooth submanifold $X\subset \ _rJ^k(\K^n,\K^p)$ such
that for all $g:\K^n\longrightarrow \K^p$ and for all
$\{z_1,\ldots,z_r\}\subset \K^n$ we have that $_rj^k
g(z_1,\ldots,z_r)\in X$ if and only if the multigerm of $g$ in
$\{z_1,\ldots,z_r\}$ is $\mathcal{A}$-equivalent to $f$. We have:

\begin{lem}
$\cod_{_rJ^k(\K^n,\K^p)} X=s+(r-1)p$.
\end{lem}
\begin{proof}
This is proved by standard multijet and transversality techniques, for a detailed account see \cite{tesis}.
\end{proof}

If the $\mathcal A$-codimension of $f_j$ is $i_j$, $j=1,\ldots,r$,
this means that each $f_j$ defines a smooth submanifold in the
appropriate jet space of respective codimension $i_j$. These
submanifolds are defined by $i_1, \ldots , i_r$ equations
respectively.

If we consider the submanifold $X\subset \ _rJ^k(\K^n,\K^p)$ defined
by the equations which define the multigerm (i.e. the equations
which define each of the branches, which are independent since they
involve different variables, plus the equations arising from all the
points having the same image in the target space), we have that its
codimension is $i_1+\ldots +i_r+(r-1)p$ (the $(r-1)p$ extra
equations come from $f(x_1)=...=f(x_r)$). From the previous Lemma
the codimension of such a submanifold is $s+(r-1)p$, so we deduce
that the $\mathcal{A}$-codimension of the multigerm is $s=i_1 +
\ldots + i_r$. In the case of some type of contact between the
strata of the discriminant of different branches, other equations
describing these contacts should be added to define the
corresponding submanifold in the multijet space and so, in that case
$s\geq i_1 + \ldots + i_r$.

When one of the branches of the multigerm is non-stable, it is not
easy to characterize the contact between the strata of the
discriminant. We need the following

\begin{definition}
Let $f:(\mathbb K^n,S)\rightarrow (\mathbb K^p,0)$ be a non-stable
germ and $F(x,\lambda)=(f_{\lambda}(x),\lambda)$ a stable unfolding
of $f$, $\lambda\in \mathbb K^m$. Let $g:(\mathbb K^n,0)\rightarrow
(\mathbb K^p,0)$ be a prism on a Morse function or an immersion such
that $\{f,g\}$ is simple. We say that $g$ is the best possible with
respect to $f$ and $F$ if
\begin{enumerate}
\item[a)]
$g$ is transverse to the limit of the tangent spaces of the strata
of the discriminant of $f$ of dimension greater than 0 and
\item[b)]
there exist representatives $F:U\times \Lambda\rightarrow \mathbb
K^{p}\times \mathbb K^m$ and $g:V\rightarrow \mathbb K^p$  of $F$
and $g$ respectively such that for almost all $0\neq\lambda\in
\Lambda$, $\{f_{\lambda},g\}:U\times V\rightarrow \mathbb K^p$ only
has stable singularities.
\end{enumerate}
\end{definition}

Notice that if $\{F,g\times id_{\mathbb K^m}\}$ is stable then
condition $b)$ holds.

\begin{ex}
\begin{enumerate}
\item[i)]
The fold map $g_1(x,y)=(x,y^2)$ is the best possible with respect to
$f(x,y)=(x^3+y^2x,y)$ and $F(x,y,\lambda)=(x^3+y^2x+\lambda
x,y,\lambda)$. However, $g_2(x,y)=(x^2,y)$ is not, since taking the
deformation $f_{\lambda}(x,y)=(x^3+y^2x+\lambda x,y)$, for
$\lambda<0$ there are two cusps of $f_{\lambda}$ lying on the
discriminant of $g_2$, and so $\{f_{\lambda},g_2\}$ has non stable
singularities.
\item[ii)]
Consider $f_{\lambda}(x,y)=(x^3+y^3x+\lambda_1 x+\lambda_2 xy,y)$.
Clearly, $g_1(x,y)=(x^2,y)$ is not the best possible with respect to
$f$ since for any value of $\lambda=(\lambda_1,\lambda_2)$ there are
either 1 or 3 cusps of $f_{\lambda}$ lying on the discriminant of
$g_1$. If we take $g_2(x,y)=(x,y^2)$, there is a cuspidal curve in
the bifurcation plane such that $f_{\lambda}$ has codimension 1
singularities (namely lips and beaks), and so, for those values of
$\lambda$, $\{f_{\lambda},g_2\}$ has non-stable singularities. Even
further, if $\lambda_1=0$, there is a cusp at $(x,y)=(0,0)$ which
lies on the discriminant of $g_2$ and again $\{f_{\lambda},g_2\}$
has non-stable singularities. However, for almost all $\lambda$,
$\{f_{\lambda},g_2\}$ only has stable singularities and so $g_2$ is
the best possible with respect to $f$ and $F$.
\item[iii)]
The fold map $g(x,y)=(x,y^2+x)$ is the best possible with respect to
the primitive germ $f$ and $F$ where
$f_{\lambda}(x,y)=(x^4+yx+\lambda x^2,y)$. Notice that $\{F,g\times
id_{\mathbb K}\}$ is not stable.
\end{enumerate}
\end{ex}

So if we have a simple germ $h=\{f,g\}$ with $f$ non stable, $F$ a
stable unfolding of $f$ and $g$ a prism on a Morse function or an
immersion which is the best possible with respect to $f$ and $F$,
then, by the above Lemma and considerations, $$\mathcal
A-\cod(h)=\mathcal A-\cod(f)+\mathcal A-\cod(g)=\mathcal
A-\cod(f)+n-p+1.$$ The fact that this is true for example $iii)$
above is an exceptional case since, as we will see in Corollary
\ref{prim1}, a multigerm composed of a non-stable primitive germ and
a fold map is almost always non-simple.

\subsection{Augmentations and concatenations}

We define the operation of simultaneous augmentation and monic
concatenation and derive a formula for the $\mathcal
A_e$-codimension of the resulting multigerm:

\begin{teo}\label{augconc}\cite{nosso}
Suppose $f:(\mathbb K^n,S)\rightarrow (\mathbb K^p,0)$ has a
1-parameter stable unfolding
$F(x,\lambda)=(f_{\lambda}(x),\lambda)$.  Let $g:(\mathbb
K^p\times\mathbb K^{n-p+1},0)\rightarrow (\mathbb K^p\times\mathbb
K,0)$ be the fold map $(X,v)\mapsto(X,\Sigma_{j=p+1}^{n+1} v_j^2)$.
Then,

i) the multigerm $\{A_{F,\phi}(f),g\}$, where $\phi:\mathbb
K\rightarrow\mathbb K$, has
$$\mathcal{A}_e-cod(\{A_{F,\phi}(f),g\})\geq\mathcal{A}_e-cod(f)(\tau(\phi)+1),$$
where $\tau$ is the Tjurina number of $\phi$. Equality is reached
when $\phi$ is quasi-homogeneous and $\langle
dZ(i^*(Lift(A_{F,\phi}(f))))\rangle=\langle dZ(i^*(Lift(F)))\rangle$
where $i:\mathbb K^p\rightarrow \mathbb K^{p+1}$ is the canonical
immersion $i(X_1,\ldots,X_{p})=(X_1,\ldots,X_{p},0)$ and $dZ$
represents the last component of the target vector fields.

ii) $\{A_{F,\phi}(f),g\}$ has a 1-parameter stable unfolding.
\end{teo}

\begin{rem}
We do not know an example where the condition $\langle
dZ(i^*(Lift(A_{F,\phi}(f))))\rangle=\langle dZ(i^*(Lift(F)))\rangle$
in the previous theorem is not satisfied. However, we do not have a
proof that it is true in general. A similar technical condition
appears in \cite[Theorem 3.8]{robertamond} for the $\mathcal
A_e$-codimension of the binary concatenation and in the definition
of substantial unfolding in \cite{houston}.
\end{rem}

We need the following:

\begin{lem}\label{adja}
Suppose $f=\{f_1,\ldots,f_r\}:(\mathbb K^n,S)\rightarrow (\mathbb
K^p,0)$ has a 1-parameter stable unfolding $F$, then we have the
following adjacency diagram between augmentations of $f$:
$$F\longleftarrow A_{F,z^{2}}(f)\longleftarrow\ldots \longleftarrow A_{F,z^{k-1}}(f)\longleftarrow A_{F,z^{k}}(f) \longleftarrow \ldots$$
\end{lem}
\begin{proof}
First suppose that $f$ can be divided into two non-stable germs
$h_1$ and $h_2$. Then $F=\{H_1,H_2\}$ where $H_i$ is a stable
unfolding of $h_i$, $i=1,2$. Since
$\dim\widetilde\tau(h_1)=\dim\widetilde\tau(h_2)=0$, we have
$\dim\widetilde\tau(H_i)\leq 1$ for $i=1,2$. Now,
$\widetilde\tau(H_1)$ and $\widetilde\tau(H_2)$ have to be
transversal because $F$ is stable, which can only happen if $p+1=2$.
However, when $p=1$, there is no such germ. This means that if $f$
has a 1-parameter stable unfolding then there is at most one branch
(say $f_1$) which is not stable and the germ $\{f_2,\ldots,f_r\}$ is
stable.

Therefore, we can assume that the unfolding parameter in $F$ appears only in $F_1$, i.e.
\begin{equation}
F(x,\lambda)=
\begin{cases}
(f_{1_\lambda}(x),\lambda)\\
(f_2(x),\lambda)\\
\ldots\\
(f_r(x),\lambda).
\end{cases}
\end{equation}

Now consider the augmentation
$A_{F,z^{k}}(f)(x,z)=\{(f_{1_{z^k}}(x),z),\ldots,(f_r(x),z)\}$. The
germ $\{(f_{1_{(z^k+uz^{k-1})}}(x),z),\ldots,(f_r(x),z)\}$ is
contained in the versal unfolding of $A_{F,z^{k}}(f)(x,z)$ and is
$\mathcal R$-equivalent to
$\{(f_{1_{z^{k-1}}}(x),z),\ldots,(f_r(x),z)\}=A_{F,z^{k-1}}(f)(x,z)$.
The result follows.
\end{proof}

\begin{teo}\label{simpleaugconc}
Suppose $f:(\mathbb K^n,S)\rightarrow (\mathbb K^p,0)$ has a
1-parameter stable unfolding
$F(x,\lambda)=(f_{\lambda}(x),\lambda)$. Let $g:(\mathbb
K^p\times\mathbb K^{n-p+1},0)\rightarrow (\mathbb K^p\times\mathbb
K,0)$ be the fold map $(X,v)\mapsto(X,\Sigma_{j=p+1}^{n+1} v_j^2)$.
Suppose that $\phi$ is quasi-homogeneous, $A_{F,\phi}(f)$ is simple
and $\langle dZ(i^*(Lift(A_{F,\phi}(f))))\rangle=\langle
dZ(i^*(Lift(F)))\rangle$, then $\mathcal{A}_e-cod(f)=1$ implies that
$\{A_{F,\phi}(f),g\}$ is simple. Furthermore, if $g$ is transverse
to the limits of the tangent spaces of the strata of
$A_{F,\phi}(f)$, then the converse is also true.
\end{teo}
\begin{proof}
From Theorem \ref{augconc} we have that
$\mathcal{A}_e-cod(\{A_{F,\phi}(f),g\})=\mathcal{A}_e-cod(f)(\tau(\phi)+1)$.

Suppose first that $\mathcal{A}_e-cod(f)=1$. We know that the
stratum codimension of $\{A_{F,\phi}(f),g\}$ is greater than or
equal to $\mathcal A_e-cod(A_{F,\phi}(f))+1=\mathcal
A_e-cod(f)\tau(\phi)+1=\tau(\phi)+1=\mathcal{A}_e-cod(\{A_{F,\phi}(f),g\})$.
The stratum codimension can never be greater than the $\mathcal
A_e$-codimension, so they must be equal. Having this, the only way
for $\{A_{F,\phi}(f),g\}$ to be non-simple is that it is an
exceptional value of the parameter of a family with modality.
Considering Lemma \ref{adja}, since $A_{F,\phi}(f)$ is simple, the
modal family would be $\{A_{F,\phi'}(f),g\}$ with
$\tau(\phi')=\tau(\phi)-1$ and clearly this is not the case.
Therefore, $\{A_{F,\phi}(f),g\}$ is simple.

Now suppose that $\{A_{F,\phi}(f),g\}$ is simple. Its normal form is
\begin{equation}
\begin{cases}
(f_{\phi(z)}(x),z)\\
(X,\Sigma_{j=p+1}^{n+1} v_j^2)
\end{cases}
\end{equation}
If we take the 1-parameter stable unfolding of the augmentation
$\widetilde F(x,z,\lambda)=(f_{\phi(z)+\lambda}(x),z,\lambda)$, it
turns out by part ii) of Theorem \ref{augconc} that
\begin{equation}
\begin{cases}
(f_{\phi(z)+\lambda}(x),z,\lambda)\\
(X,\Sigma_{j=p+1}^{n+1} v_j^2,\lambda)
\end{cases}
\end{equation}
is a 1-parameter stable unfolding of $\{A_{F,\phi}(f),g\}$.
Therefore, if we consider the deformation
$\{(f_{\phi(z)+\lambda}(x),z),(X,\Sigma_{j=p+1}^{n+1} v_j^2)\}$, it
only has stable singularities. Since $g$ has no contact with the
strata of $A_{F,\phi}(f)$, $g$ is the best possible with respect to
$A_{F,\phi}(f)$ and $\widetilde F$ and so $$\mathcal
A-\cod(\{A_{F,\phi}(f),g\})=\mathcal A-\cod(A_{F,\phi}(f))+\mathcal
A-\cod(g)=\mathcal A-\cod(A_{F,\phi}(f))+n-p+1.$$

Wilson's formula yields
\begin{align}
\mathcal{A}_e-cod(\{A_{F,\phi}(f),g\})&=\mathcal A-\cod(\{A_{F,\phi}(f),g\})+(r+1)(p-n)-p\\
&=\mathcal A-\cod(A_{F,\phi}(f))+n-p+1+(r+1)(p-n)-p\\
&=\mathcal A_e-\cod(A_{F,\phi}(f))+1\\
&=\mathcal A_e-\cod(f)\tau(\phi)+1.
\end{align}
On the other hand, since
$\mathcal{A}_e-cod(\{A_{F,\phi}(f),g\})=\mathcal{A}_e-cod(f)(\tau(\phi)+1)$,
$\mathcal A_e-\cod(f)=1$.
\end{proof}

\begin{ex}\label{exaugconc}
\begin{enumerate}
\item[i)] Let $f(y)=(y^2,y^3)$ and consider the augmentations and concatenations
\begin{equation}
\begin{cases}
(y^2,y^3+x^{k+1}y,x)\\
(y,x,0)
\end{cases}
\end{equation}
These bigerms are called $A_0S_k$ ($k\geq 1$) in \cite{hobbskirk}
and \cite{roberta} and are simple.
\item[ii)] Let $f(y)=(y^2,y^{5})$ and consider the augmentation and concatenation
\begin{equation}
\begin{cases}
(y^2,y^{5}+x^2y,x)\\
(y,x,0)
\end{cases}
\end{equation}
The bigerm $A_0B_2$ is not simple since $\mathcal A_e-\cod(f)=2$ and
the immersion is transverse to the strata of $B_2$. Therefore, the
bigerms $A_0B_k$ are not simple for $k>1$.
\item[iii)] Consider the codimension 1, $n$-germ from $\mathbb K^{n-1}$ to $\mathbb K^{n-1}$
\begin{equation}
\begin{cases}
(x_1^2,x_2,\ldots,x_{n-1})\\
\ldots\\
(x_1,x_2,\ldots,x_{n-1}^2)\\
(x_1^2+x_2+\ldots+x_{n-1},x_2,\ldots,x_{n-1})\\
\end{cases}
\end{equation}
and augment and concatenate to obtain the $n+1$-germ from
$\K^n\rightarrow\K^n$
\begin{equation}
\begin{cases}
(x_1^2,x_2,\ldots,x_{n-1},z)\\
\ldots\\
(x_1,x_2,\ldots,x_{n-1}^2,z)\\
(x_1^2+x_2+\ldots+x_{n-1}+\phi(z),x_2,\ldots,x_{n-1},z)\\
(x_1,x_2,\ldots,x_{n-1},z^2)
\end{cases}
\end{equation}
If $\phi$ is quasihomogeneous, $\phi(z)=z^k$ and we obtain a simple
multigerm of codimension $k$. This means that there are infinitely
many simple multigerms with $n+1$ fold branches. However, as we will
se later, there is no simple multigerm with $n+2$ branches. We
remark here that by \cite[Corollary 3.9]{nosso}, any multigerm with
$n+1$ fold branches is an augmentation and concatenation. These
examples also hold for the case $(n,n+1)$ considering immersions
instead of folds.
\item[iv)]
Consider the codimension 1, $n-1$-germ from $\mathbb K^{n-2}$ to
$\mathbb K^{n-2}$ and augment and concatenate it twice. We obtain
infinitely many non-simple multigerms from $\mathbb K^n$ to $\mathbb
K^n$ with $n+1$ fold branches of codimension
$(\tau(\phi_1)+1)(\tau(\phi_2)+1)$. The last fold is transverse to
the strata of the previous $n$-germ.
\begin{equation}
\begin{cases}
(x_1^2,x_2,\ldots,x_{n-2},y,z)\\
\ldots\\
(x_1,x_2,\ldots,x_{n-2}^2,y,z)\\
(x_1^2+x_2+\ldots+x_{n-2}+\phi_1(y)+\phi_2(z),x_2,\ldots,x_{n-2},y,z)\\
(x_1,x_2,\ldots,x_{n-2},y^2,z)\\
(x_1,x_2,\ldots,x_{n-2},y,z^2)\\
\end{cases}
\end{equation}
\item[v)]
The extra hypothesis for the converse of Theorem \ref{simpleaugconc}
to be true is necessary. If we simultaneously augment and
concatenate the codimension 2 bigerm $\{(x^2,y),(x^2+y^3,y)\}$ we
obtain the codimension 4 simple trigerm (\cite{roberta})
\begin{equation}
\begin{cases}
(x^2,y,z)\\
(x^2+y^3+z^2,y,z)\\
(x,y,z^2)\\
\end{cases}
\end{equation}
Notice that the double point curve for
$\{(x^2,y,z),(x^2+y^3+z^2,y,z)\}$ describes a cusp which is tangent
in the limit to $g$.
\end{enumerate}
\end{ex}

\subsection{Generalised concatenations}

Now we study the simplicity of multigerms admitting a decomposition
$h=\{f,g\}$ where $f$ and $g$ are stable germs. We prove in
Proposition \ref{0n-2} that if $\widetilde\tau(f)=\{0\}$ and
$\dim_{\mathbb K}\widetilde\tau(g)=p-2$, then $h$ is not simple.
From this we deduce in Corollary \ref{maingenconc} that generalised
concatenations where $\widetilde\tau(f)=\{0\}$ are non-simple.
Furthermore, we discuss simplicity of $h$ when $1\leq
\dim\widetilde\tau(f),\dim\widetilde\tau(g)<p-1$, which may or may
not be generalised concatenations.

\begin{definition}\cite{nosso}
Let $f:(\mathbb K^{n-s},S)\rightarrow (\mathbb K^{p-s},0)$, $s<p$,
be a germ of finite $\mathcal{A}_e$-codimension and let $F:(\mathbb
K^{n},S\times \{0\})\rightarrow (\mathbb K^{p},0)$ be a
$s$-parameter stable unfolding of $f$ with
$$F(x_1,\ldots,x_n)=(F_1(x_1,\ldots,x_n),\ldots,F_{p-s}(x_1,\ldots,x_n),x_{n-s+1},\ldots,x_n),$$
where $F_i(x_1,\ldots,x_{n-s},0,\ldots,0)=f_i(x_1,\ldots,x_{n-s})$.
Suppose that $\overline g:(\mathbb K^{n-p+s},T)\rightarrow (\mathbb
K^{s},0)$ is stable. Then the multigerm $h=\{F,g\}$ is a generalised
concatenation of $f$ with $g$, where $g=Id_{\mathbb K^{p-s}}\times
\overline g$.
\end{definition}

Observe that with this definition, $\dim\widetilde\tau(g)\geq
p-s\geq 1$. If $g$ is a monogerm and $\dim\widetilde\tau(g)=p-s$, it
is of the form
$$g(x_1,\ldots,x_n)=(x_1,\ldots,x_{p-s},g_{p-s+1}(x_{p-s+1},\ldots,x_n),\ldots,g_{p}(x_{p-s+1},\ldots,x_n)).$$

When $s=1$ and $g_p(x_p,\ldots,x_n)=\Sigma_{i=p}^{n}x_i^2$ (or
$g_p=0$ when $n=p-1$), $h$ is called a monic concatenation. When $h$
is of the form
\begin{equation}
\begin{cases}
(X,y,u)\mapsto (f_u(y),u,X)\\
(x,Y,u)\mapsto (Y,u,g_{u}(x))
\end{cases}
\end{equation}
where $(f_u(y),u)$ and $(u,g_u(x))$ are 1-parameter stable
unfoldings of a certain $f$ and $g$ respectively, $h$ is called a
binary concatenation.

In \cite{nishisimple}, Nishimura proved the following Theorem:

\begin{teo}\label{nishi}
Let $f=\{f_1,\ldots,f_r\}:(\K^n,S)\rightarrow (\K^p,0)$ with $n\leq
p$ be a multigerm with minimal corank. If $np\neq 1$ and $f$ is
$\mathcal A$-simple, then the following inequality holds
$$m_0(f)\leq \frac{p^2+(n-1)r}{n(p-n)+n-1}.$$
\end{teo}

From this result we obtain

\begin{prop}\label{0n-2}
Let $h=\{f,g\}$ be a multigerm with $f,g$ stable and $n=p\neq 1,2$
or $n=p-1$. Suppose that $\widetilde\tau(f)=\{0\}$ and
$\dim_{\mathbb K}\widetilde\tau(g)=p-2$, then $h$ is not $\mathcal
A$-simple.
\end{prop}
\begin{proof}
Suppose that $h$ is simple.

1) First take the case $n=p$. From Nishimura's result we have that
$m_0(h)\leq \frac{n^2+(n-1)r}{n-1}.$ Since $f$ is stable and
$\widetilde\tau(f)=\{0\}$, it must be an
$A_{k_1,\ldots,k_s}$-singularity with $\sum_{i=1}^s k_i=n$. On the
other hand $\dim_{\mathbb K}\widetilde\tau(g)=n-2$ implies that $g$
is either an $A_2$-singularity or an $A_1^2$-singularity. We have
that $$m_0(h)= m_0(f)+m_0(g)=\sum_{i=1}^s
(k_i+1)+m_0(g)=n+s+m_0(g),$$ where $m_0(g)=3$ or 4 depending on
wether $g$ is an $A_2$ or an $A_1^2$, respectively.

For the $A_2$ case we have that
$n+s+3\leq\frac{n^2+(n-1)(s+1)}{n-1}$ where $s+1=r$. This implies
that $n-2\leq 0$ and therefore $n=1,2$. For example the bigerms
$\{x^2,x^3\}$ when $n=1$ and $\{(x^3+xy,y),(x,y^3+xy)\}$ when $n=2$
are simple (\cite{nosso}).

In the $A_1^2$ case $n+s+4\leq\frac{n^2+(n-1)(s+2)}{n-1}$ where
$s+2=r$. Again this implies that $n=1,2$. For example the trigerms
$\{x^2,x^2,x^2\}$ when $n=1$ and $\{(x^3+xy,x),(x,y^2),(x,y^2+x)\}$
when $n=2$ are simple (\cite{nosso}).

2) For the case $n=p-1$, Nishimura's result yields $m_0(h)\leq
\frac{p^2+(p-2)s}{2p-3}.$  Here $\dim_{\mathbb
K}\widetilde\tau(g)=p-2$ implies that $g$ is a transversal
intersection of two immersions. We distinguish between the cases
where $n$ is even or odd.

If $n$ is even, $\widetilde\tau(f)=\{0\}$ implies that $f$ is a
monogerm with $m_0(f)=\frac{n+2}{2}$ or it is a $p$-tuple point with
$m_0(f)=p$. If $f$ is a monogerm we have that
$\frac{n+2}{2}+2=\frac{p+5}{2}\leq \frac{p^2+3(p-2)}{2p-3}$, which
implies $p\leq 3$, however when $p=3$ a cross-cap together with two
immersions is not simple (\cite{roberta}) so $h$ is not simple. If
$f$ is a $p$-tuple point we have that $p+2\leq
\frac{p^2+(p-2)(p+2)}{2p-3}$ and so $p\leq 2$, which is a
contradiction.

If $n$ is odd, $\widetilde\tau(f)=\{0\}$ implies that $f$ is either
a bigerm $\{f_1,f_2\}$ with $m_0(f_1)=\frac{n-1+2}{2}$ and
$m_0(f_2)=1$ or it is a $p$-tuple point. The case where $h$ is a
$p+2$-tuple point is the same as in the case that $n$ is even and
yields $p\leq 2$, however, the cross-ratio shows that a quadruple
point when $p=2$ is not simple. When $f$ is a bigerm we get the
inequality $\frac{n+1}{2}+1+2=\frac{p+6}{2}\leq
\frac{p^2+4(p-2)}{2p-3}$ which again implies $p\leq 2$.
\end{proof}

\begin{ex}\label{ejem0n-2}
\begin{enumerate}
\item[i)]
In the equidimensional case, the bigerm $A_2A_n$ from $\K^n$ to
$\K^n$ given by
\begin{equation}
\begin{cases}
(x_1^{n+1}+x_2x_1+\ldots +x_{n-2}x_1^{n-3}+yx_1^{n-2}+zx_1^{n-1},x_2,\ldots,x_{n-2},y,z)\\
(x_1,\ldots,x_{n-2},y,z^3+yz)
\end{cases}
\end{equation}
is not simple when $n>2$. It has $\mathcal A_e$-codimension $n$
(\cite{nosso}) but the stratum codimension is always 2.

\item[ii)] A $p+2$-tuple point for any $(n,p)$ with $n\geq p-1$ is not simple.
\end{enumerate}
\end{ex}

\begin{coro}\label{morse}
Let $h=\{f,g\}$ be a multigerm with $f,g$ stable and
$\widetilde\tau(f)=\{0\}$. If $h$ is simple, then $g$ is a prism on
a Morse function or an immersion.
\end{coro}

\begin{coro}\label{maingenconc}
Let $h=\{f,g\}$ be a non-monic generalised concatenation (i.e. $g$
is not a prism on a Morse function or an immersion) and suppose that
$\widetilde\tau(f)=\{0\}$, then $h$ is non simple.
\end{coro}

The case $h=\{f,g\}$ with $f$ and $g$ stable and $1\leq
\dim\widetilde\tau(f),\dim\widetilde\tau(g)<p-1$ is not included in
the above results. Suppose $h$ is of type $A_{k_1\ldots k_r}$ from
$\K^n$ to $\K^n$ where $\sum_{i=1}^{r}k_i=n+1$. This implies that
$m_0(h)=\sum_{i=1}^{r}k_i+r=n+r+1<n+r+1+\frac{1}{n-1}=\frac{n^2+r(n-1)}{n-1}$,
which means that the multiplicity of such a multigerm is the maximum
possible below Nishimura's bound. We have the following

\begin{prop}
There exists a simple $h:(\K^n,S)\rightarrow (\K^n,0)$ of type
$A_{k_1\ldots k_r}$ with $\sum_{i=1}^{r}k_i=n+1$.
\end{prop}
\begin{proof}
We can decompose $h$ in two stable germs $A_{k_{i_1}\ldots k_{i_s}}$
and $A_{k_{j_1}\ldots k_{j_{r-s}}}$ such that $k_{i_1}+\ldots
+k_{i_s}=l$ and $k_{j_1}+\ldots +k_{j_{r-s}}=n+1-l$. There exist
germs of type $A_{k_{i_1}\ldots k_{i_s}}$ and $A_{k_{j_1}\ldots
k_{j_{r-s}}}$ which have codimension 1 as germs in $\K^{l-1}$ and
$\K^{n-l}$ (\cite{robertamond}). With them we can construct a
codimension 1 binary concatenation which is of type $A_{k_1\ldots
k_r}$ in $\K^n$ and is therefore simple.
\end{proof}

A similar study can be done for the case of multigerms
$h:(\K^n,S)\rightarrow (\K^{n+1},0)$ where the multiplicity is the
maximum possible below Nishimura's bound
$N=\frac{(n+1)^2+r(n-1)}{2n-1}$.

Suppose $n=2l+1$. Since  $r\leq n+2=2l+3$ then for $(l,r)\neq
(1,2)$, $N=\frac{(2l+2)^2+2lr}{4l+1}\leq l+1+\frac{r}{2}$ when $r$
is even or $N\leq l+1+ \frac{r+1}{2}$ when $r$ is odd. In fact
$N=l+2+\frac{r}{2}-\frac{2l-4+r}{2(4l+1)}$ and
$0<\frac{2l-4+r}{2(4l+1)}\leq\frac{1}{2}$. If $l=1$ and $r=2$ then
$N=4$ and from \cite{cati} there is no simple bigerm $h=\{f,g\}$
with $f,g$ stable of multiplicity 4.

Now suppose $n=2l$. Then for $l\neq 1$ and $(l,r)\neq (2,1)$,
$N=\frac{(2l+1)^2+(2l-1)r}{4l-1}\leq [l+1+\frac{r}{2}]$. In fact
$N=l+1+\frac{r}{2}+\frac{l+2}{4l-1}-\frac{r}{2(4l+1)}$ and
$\frac{l+2}{4l-1}-\frac{r}{2(4l+1)}<\frac{1}{2}$. When $l=2$ and
$r=1$, there is not a stable monogerm of multiplicity 4. Suppose
$l=1$. If $r=3$ then $N=4$ and from \cite{roberta} the only simple
trigerms are those composed by 3 immersions and therefore have
multiplicity 3. If $r=2$ then $N=3$ and there are simple bigerms
whose branches are a cross-cap and an immersion (\cite{roberta}).
We have the following

\begin{prop}
There exists a simple multigerm $h:(\K^n,S)\rightarrow (\K^{n+1},0)$
with $m_0(h)=l+1+\frac{r}{2}$ when $n=2l$ or $n=2l+1$ and $r$ is
even, or $m_0(h)=l+1+\frac{r+1}{2}$ when $n=2l+1$ and $r$ is odd.
\end{prop}
\begin{proof}
First suppose that $n=2l+1$ and $r$ is even. We can write
$h=\{f,g\}$ such that $m_0(f)=l+1$ and is stable. Notice that $1\leq
\dim\widetilde\tau(f)\leq l+1$.  Consider $g$ the multigerm of
$\frac{r}{2}$ immersions with $m_0(g)=\frac{r}{2}$ and take  $f$
with $\frac{r}{2}$ branches. Then $h$ has the desired multiplicity
and number of branches and is stable since the analytic strata have
regular intersection. In fact,
$\cod\widetilde\tau(f)=2l+2-\dim\widetilde\tau(f)=2l+2-\frac{r}{2}$
and
$\cod\widetilde\tau(g)=\cod\widetilde\tau(A_0^{\frac{r}{2}})=\frac{r}{2}$,
so we can always choose them in a way that they have regular
intersection. Obviously, any stable germ in the nice dimensions is
simple.

If $n=2l$ and $r$ is even then there exists  a codimension 1 germ
whose versal unfolding is the germ $h$ constructed above and
therefore is simple.

Now suppose $n=2l+2$ and $r$ is odd. Then
$[(l+1)+1+\frac{r}{2}]=l+1+\frac{r+1}{2}$. We can write $h=\{f,g\}$
such that $m_0(f)=l+1$ and is stable. This means that $2\leq
\dim\widetilde\tau(f)\leq l+2$. Similarly to the previous case,
consider $g$ the multigerm of $\frac{r+1}{2}$ immersions with
$m_0(g)=\frac{r+1}{2}$ and take  $f$ with $\frac{r-1}{2}$ branches,
then $h$ has the desired multiplicity and number of branches and is
stable since the analytic strata have regular intersection and
therefore simple.

If $n=2l+1$ and $r$ is odd then there exists  a  codimension 1 germ
whose versal unfolding is the germ $h$ constructed above and is
therefore simple.
\end{proof}

However, there are examples of multigerms with the highest possible
multiplicity below Nishimura's bound that are not simple:

\begin{ex}\label{exdosstable}
\begin{enumerate}
\item[i)]
Consider the codimension 2 trigerm
$\{(x^2,y),(x,y^2),(x,y^2+x^2)\}$. If we augment and concatenate it
we obtain the codimension 4 quadrigerm
\begin{equation}
\begin{cases}
(x^2,y,z)\\
(x,y^2,z)\\
(x,y^2+x^2+z^2,z)\\
(x,y,z^2)
\end{cases}
\end{equation}
If we take the first two branches as $f$ and the last two as $g$ we
have that $1=\dim\widetilde\tau(f)=\dim\widetilde\tau(g)$ but this
multigerm is not simple by Theorem \ref{simpleaugconc}. The same
example is valid for $(n,p)=(2,3)$ considering immersions instead of
folds.
\item[ii)]
Suppose we have a germ of type $A_{k_1\ldots k_r}$ from $\K^n$ to
$\K^n$ such that $\sum_{i=1}^{r}k_i=n+1$ and $k_{r-1}=k_r=1$. Since
$\sum_{i=1}^{r-2}k_i=n-1$, there exists a germ of type $A_{k_1\ldots
k_{r-2}}$ which has codimension 1 in $\K^{n-2}$. We can augment and
concatenate it with an augmenting function $\phi$ such that
$\tau(\phi)=t>1$ to obtain a germ in $\K^{n-1}$ of codimension
$t+1>2$. If we augment and concatenate this germ again we obtain a
non simple germ of type $A_{k_1\ldots k_r}$, provided the last fold
$A_{k_r}=A_1$ is transverse to the strata of the germ of type
$A_{k_1\ldots k_{r-1}}$.

\item[iii)]
By Theorem \ref{nishi}, two $A_{n-1}$ singularities in $\K^n$ are
simple only when $n\leq 3$ (two cuspidal edges, for example).
\end{enumerate}
\end{ex}

It follows by \cite{nishisimple} that if
$f:(\K^n,S)\rightarrow(\K^p,0)$ ($n\leq p$) is simple, then the
number of branches $r$ is bounded by $\frac{p^2}{n(p-n)}$. In the
equidimensional case this is not an upper bound. However, if we
consider only non-submersive branches we can prove the following.

\begin{prop}
Let $f=\{f_1,\ldots,f_r\}:(\K^n,S)\rightarrow (\K^n,0)$ be a germ of
type $A_{k_1,\ldots,k_r}$ with $|S|=r>1$ and $n\geq k_i\geq k_{i+1}$
$\forall i=1,\ldots,r-1$. If $f$ is simple, then $r\leq
n-k_1+2=n-m_0(f_1)+1$.
\end{prop}
\begin{proof}
If $k_1=1$, then all the other branches are also fold singularities.
From Example \ref{ejem0n-2}, a simple multigerm with only fold
singularities can have at most $n+1$ branches.

If $k_1=2$,  from Proposition \ref{0n-2} since $f$ is simple then
$\dim \widetilde\tau(f')>0$ where $f'=\{f_2,\ldots,f_r\}$. Therefore
$f'$  has at most $n-1$ branches, and so $r\leq n$. In fact the best
multigerm  that has  analytic stratum zero is the $n$-tuple
transversal point.

If $k_1=k\leq n$, then $\dim\widetilde\tau(f_1)=n-k$. In the best of
the cases, the remaining branches are folds. Suppose we take $n-k$
transversal folds whose intersection has dimension $k$. Then
$\widetilde\tau(\{f_1,A_1^{n-k}\})=\{0\}$, and so, by Corollary
\ref{morse}, there is just one more branch and is  a prism on a
Morse function. Therefore $r\leq n-k+1+1=n-k+2$.
\end{proof}

\subsection{Multigerms with a non-stable branch}

We study here germs $h=\{f,g\}$ where $f$ is a non-stable primitive
germ. We classify all simple germs where $g$ is a prism on a Morse
function or an immersion and give some results for the general case.

\begin{coro}\label{prim1}
Let $f=\{f_1,\ldots,f_r\}:(\K^n,S)\rightarrow (\K^n,0)$ be a
primitive $\mathcal A_e$-codimension 1 germ, $n>2$. Then the
multigerm $h=\{f,A_1\}$ is not simple.
\end{coro}
\begin{proof}
If $f$ is a multigerm, from \cite{robertamond} $f_i$ is stable for
all $i=1,\ldots,r$, so $h=A_{k_1,\ldots,k_{r},1}$ and
$m_0(h)=n+1+r+2$. If $f$ is a monogerm, $m_0(f)=n+2$
(\cite{riegerruas}) and $m_0(h)=n+2+2=n+1+r+2$. Suppose that $h$ is
simple. By Nishimura's result $n+r+3\leq\frac{n^2+(n-1)(r+1)}{n-1}$
and so $n\leq 2$.
\end{proof}

\begin{coro}\label{prim2}
Let $f:(\K^n,0)\rightarrow (\K^{n+1},0)$ be a primitive $\mathcal
A_e$-codimension 1 germ, $n>3$. Then the multigerm $h=\{f,A_0\}$ is
not simple.
\end{coro}
\begin{proof}
From \cite{robertamond} we know that $m_0(f)=\frac{n+3}{2}$ and that
$n$ is odd, since there are no primitive $\mathcal A_e$-codimension
1 when $n$ is even. Suppose that $h$ is simple. By Nishimura's
result $\frac{n+3}{2}+1\leq\frac{(n+1)^2+2(n-1)}{2n-1}$ and so
$n\leq 3$.
\end{proof}

These results can be deduced from the proof of \cite[Propostion
5.9]{nosso} which states that if $h=\{f,g\}$ is a multigerm with $f$
a primitive monogerm of $\mathcal A_e$-codimension 1 and $g$ a prism
on a Morse function or an immersion, then $h$ has codimension
greater than or equal to $p$ when $n\geq p$ and greater than or
equal to $\frac{p}{2}$ when $n=p-1$.

\begin{ex}\label{ejemcati}
\begin{enumerate}
\item[i)]
When $p=1$, the bigerm of a Morse function and an $A_2$-singularity
and the trigerm of 3 Morse functions have codimension 2 and are
simple.
\item[ii)]
If $n=1$ and $p=2$, there is the simple codimension 2 bigerm
$\{(x^2,x^3),(0,x)\}$, and if $n=p=2$ there are the simple
codimension 2 bigerm
\begin{equation}
\begin{cases}
(x^4+yx,y)\\
(x,y^2+x)
\end{cases}
\end{equation}
and the trigerm
\begin{equation}
\begin{cases}
(x^3+xy,x)\\
(x,y^2)\\
(x,y^2+x)
\end{cases}
\end{equation}

\item[iii)]
In the equidimensional case, given the bigerm
\begin{equation}
\begin{cases}
(x_1^{n+2}+x_2x_1+\ldots+x_nx_1^{n-1},x_2,\ldots,x_n)\\
(x_1,\ldots,x_{n-1},x_{n}^2+x_{n-1})
\end{cases}
\end{equation}
the codimension is exactly $n$ (except when $n=1$, see case 1)
above) and is non-simple when $n>2$.

\item[iv)] When $(n,p)=(3,4)$, the bigerm
\begin{equation}
\begin{cases}
(u,v,x^3+ux,x^4+vx)\\
(u,u,v,x)
\end{cases}
\end{equation}
has codimension 2 and is simple (\cite{cati}). There are no
primitive codimension 1 multigerms in these dimensions.

\item[v)]
When $(n,p)=(2,3)$, a cross-cap and two immersions or a quintuple
point are not simple (\cite{hobbskirk}, \cite{roberta}).
\end{enumerate}
\end{ex}

\begin{teo}
Let $h=\{f,g\}$ is a multigerm with $f$ a non stable germ and $g$ a
prism on a Morse function or an immersion and suppose that $g$ is
transverse to the limits of the tangent spaces of $f$. Then $h$ is
simple if and only if either $f$ is an augmentation of an $\mathcal
A_e$-codimension 1 germ (i.e. $h=\{A_{P,\phi}(p),g\}$ with $\mathcal
A_e-\cod(p)=1$) or $h$ is one of examples $i),ii)$ or $iv)$ above.
\end{teo}
\begin{proof}
Follows directly from Theorem \ref{simpleaugconc} and Corollaries
\ref{prim1} and \ref{prim2}.
\end{proof}

Example \ref{ejemcati} shows that simple multigerms $h=\{f,g\}$
where $f$ is a primitive monogerm and $g$ is a prism on a Morse
function or an immersion are exceptional. We expect that if $g$ is a
more degenerate stable singularity, $h$ will not be simple. In what
follows we discuss the case where $f$ is an augmentation and $g$ is
more degenerate than prism on a Morse function or an immersion.

\begin{coro}
Let $A_{F,\phi}(f):(\mathbb K^n,0)\rightarrow (\mathbb K^p,0)$ be an
augmentation and $g$ be a cuspidal edge or two transversal folds
(when $n\geq p$) or two transversal immersions (when $n=p-1$). If
$m_0(A_{F,\phi}(f))>\frac{n^2-n+1}{n-1}$ (when $n\geq p$) or
$m_0(A_{F,\phi}(f))>\frac{n^2+n}{2n-1}$ (when $n=p-1$) then the
multigerm $\{A_{F,\phi}(f),g\}$ is non simple.
\end{coro}
\begin{proof}
Suppose $h=\{A_{F,\phi}(f),g\}$ is simple. First suppose $n=p$, by
Nishimura's result $m_0(h)\leq\frac{n^2+r(n-1)}{n-1}$. If $g$ is a
cuspidal edge we have
$$m_0(A_{F,\phi}(f))+m_0(g)=m_0(A_{F,\phi}(f))+3=m_0(h)\leq\frac{n^2+2(n-1)}{n-1},$$
which implies $m_0(A_{F,\phi}(f))\leq\frac{n^2+r(n-1)}{n-1}$. The
case where $g$ is two transversal folds follows similarly by using
$m_0(g)=4$ and $r=3$.

If $n=p-1$, then $r=3$ and $m_0(g)=2$, and the result follows similarly.
\end{proof}

\begin{ex}\label{exaugcusp}
\begin{enumerate}
\item[i)] The bigerms
\begin{equation}
\begin{cases}
(x^3+(y^2+z^l)x,y,z)\\
(x,y,z^3+yz)
\end{cases}
\end{equation}
have codimension $l+1$ and are simple. The versal unfolding can be
obtained similarly to the proof of Theorem 4.12 in \cite{nosso}.

\item[ii)] The trigerms
\begin{equation}
\begin{cases}
(x,y,z^2)\\
(x,y,z^2+y^2+x^l)\\
(x^3+yx,y,z)
\end{cases}
\end{equation}
have codimension $l+1$ and are simple, by the same argument as above.

\item[iii)] The trigerms
\begin{equation}
\begin{cases}
(x^3+(y^2+z^l)x,y,z)\\
(x,y^2,z)\\
(x,y,z^2)
\end{cases}
\end{equation}
are augmentation and concatenation of the codimension 2 bigerm
$\{(x^3+y^2x,y),(x,y^2)\}$ and so have codimension $2l$ and are
non-simple.

\item[iv)] The bigerms
\begin{equation}
\begin{cases}
(x^4+yx+z^lx,y,z)\\
(x,y,z^3+yz)
\end{cases}
\end{equation}
are non simple.
\end{enumerate}
\end{ex}

\section{Simple multigerms from $\mathbb C^3$ to $\mathbb C^3$}

In this section we obtain a list which includes all simple
multigerms from $\mathbb C^3$ to $\mathbb C^3$ using our results and
some simple calculations.

\subsection{Monogerms}

The following table, obtained by W. L. Marar $\&$ F. Tari in
\cite{marartari} and earlier by V. Goryunov in \cite{Goryunov},
contains a list of normal forms for simple corank 1 monogerms of
maps from $\R^3$ to $\R^3$.

\hspace{1cm}\begin{tabular}{| c | c | c |} \hline Name & Normal
form & $\mathcal{A}_e$-codimension \\ \hline
$A_1$ & $(x,y,z^2)$ & 0 \\ \hline $3_{\mu(P)}$ & $(x,y,z^3+P(x,y)z)$ & $\mu(P)$ \\
\hline $4_1^k$ & $(x,y,z^4+xz \pm y^k z^2),k\geq1$ & $k-1$ \\
\hline $4_2^k$ &
$(x,y,z^4+(y^2 \pm x^k)z+x z^2),k\geq2$ & $k$ \\
\hline $5_1$ & $(x,y,z^5+xz+ y z^2)$ & $1$ \\ \hline $5_2$ &
$(x,y,z^5+xz+ y^2 z^2+y z^3)$ & $2$ \\
\hline
\end{tabular} \\ \\

Here $P(x,y)$ are simple functions in two variables and $\mu(P)$
denotes the Milnor number of $P$.  We use the standard notation
$A_2$ for the cuspidal edge $3_0$ and $A_3$ for the swallowtail
$4_1^1$.

\subsection{Bigerms}

We consider bigerms $h=\{f,g\}$.

We study first the case where $f$ is non stable. Suppose $f$ is an
augmentation and $g$ is a fold singularity $A_1$. When $h$ is an
augmentation and concatenation and from Theorem \ref{simpleaugconc}
we know that if $f$ is an augmentation of a codimension 1 germ then
$h$ is simple. So augmenting the codimension 1 germs $(x^3+y^2x,y)$
and $(x^4+yx,y)$ we obtain the families of simple germs $3_{\mu}A_1$
with $P$ an $A_{\mu}$ singularity and $4_1^kA_1$:

\begin{equation}
\begin{cases}
(x^3+(y^2+z^{\mu+1})x,y,z)\\
(x,y,z^2)
\end{cases}
\text{ and }
\begin{cases}
(x^4+yx+z^kx^2,y,z)\\
(x,y,z^2)
\end{cases}
\end{equation}

If we augment and concatenate the codimension $\mu$ germ
$(x^3+z^{\mu+1}x,z)$, since $(x,y^2,z)$ is not transversal to the
limits of the strata of the augmentations, we must consider
$\{(x^3+(y^2+z^{\mu+1})x,y,z),(x,y^2,z)\}$.

The $3_{\mu}$ cases where $P$ is a $D_k$ or $E_i$ singularity with
$k\geq 4$ and $i=6,7,8$ can be seen as augmentations of the
codimension 2 germ $(x^3+y^3x,y)$ and the $4_2^k$ cases can be seen
as augmentations of the codimension 2 germ $(x^4+y^2x+yx^2,y)$. In
all these cases, $(x,y,z^2)$ is transverse to the corresponding
strata, so the corresponding bigerm is not simple.

There are no simple germs in this case when $h$ is not an
augmentation and concatenation.

Suppose $g$ is not a fold singularity. Since $m_0(f)\geq 3$, from
Nishimura's bound we have that $m_0(g)\leq 3$ so the only
possibilities are $3_{\mu}A_2$ singularities. Following the
calculations in Example \ref{exaugcusp} i) and the fact that
$3_{\mu}A_1$ is not simple if $P$ is not an $A_{\mu}$ singularity,
these bigerms are only simple when the function $P$ in $3_{\mu}$ has
an $A_{\mu}$ singularity.

If $f$ is primitive, from Corollary \ref{prim1}, there are no simple
bigerms.

Now suppose that $f$ and $g$ are stable. First suppose that both are
$A_1$ singularities. From \cite[Proposition 3.7 and Corollary
3.8]{nosso}, $h$ must be an augmentation. It is well known that a
bigerm with two fold singularities is simple if and only if they are
transversal ($A_1^2$) or they have a simple contact (the contact
function is simple). The only possibilities are

\begin{equation}
\begin{cases}
(x,y,z^2)\\
(x,y,z^2+h(x,y))
\end{cases}
\end{equation}
where $h(x,y)$ is a simple function singularity.

We need the following Lemma to proceed which is an equidimensional
version of a Theorem in \cite{roberta}

\begin{lem}\label{lemrob}
Let $h=\{f,g\}$ and $h'=\{f',g\}$ be finitely determined germs.
Consider $_V\mathcal K$ the subgroup of the group $\mathcal K$ whose
diffeomorphism in the source preserves $V$, where $V$ is the
discriminant of $g$. If $h$ and $h'$ are $\mathcal A$-equivalent
then $\lambda$ is $_V\mathcal K$-equivalent to $\lambda'$, where
$\lambda,\, \lambda'\in \mathcal O_{3}$ are reduced defining
equations for the discriminants of $f$ and $f'$ respectively.
\end{lem}
\begin{proof}
Since $h$ and $h'$ are $\mathcal A$-equivalent there exist germs of
diffeomorphisms such that $\psi\circ f\circ \varphi =f'$ and
$\psi\circ g\circ \phi =g$. Let $D(f)$ denote the discriminant of
$f$. Then $\psi$ preserves $D(g)$ and takes $D(f)$ into $D(f')$. So
$(\lambda'\circ\psi)^{-1}(0)=\lambda^{-1}(0)$ as they are reduced
equations. Therefore $\lambda'\circ\psi$ is $\mathcal C$-equivalent
to $\lambda$ and as $\psi$ preserves $V=D(g)$, they are $_V\mathcal
K$-equivalent.
\end{proof}

By this lemma we deduce that if the function $\lambda$ defining the
discriminant of a fold singularity $f$ is non-simple, then $h$ will
be non-simple. We continue our discussion.

If $f$ is an $A_1$ singularity and $g$ is an $A_2$ singularity,
again all such bigerms are augmentations. Using the classification
of simple submersions preserving a cuspidal edge carried out in
\cite{yofarid} we obtain a list of all possible simple bigerms with
a fold and a cuspidal edge. In fact, these are all obtained by
augmenting the codimension 1 and 2 bigerms $\{(x^3+yx,y),(x,y^2)\}$
and $\{(x^3+yx,y),(x^2,y)\}$. We obtain the families

\begin{equation}
\begin{cases}
(x^3+yx,y,z)\\
(x,y^2+z^k,z)
\end{cases}
\text{ and }
\begin{cases}
(x^3+yx,y,z)\\
(x^2+z^k,y,z)
\end{cases}
\end{equation}
The case $k=1$ in both families is the stable germ $A_1A_2$.

In \cite{brucekirk}, the authors obtain a classification of
submersions under $_{V}\mathcal R$-equivalence, where $V$ is the
discriminant of the swallowtail. Similarly we can obtain the
classification of submersions under $_{V}\mathcal K$-equivalence.
The possible simple bigerms with an $A_1$ and an $A_3$ singularity:

\begin{equation}
\begin{cases}
(x^4+yx+zx^2,y,z)\\
(x,y,z^2)
\end{cases}
\text{ and for $k\geq 2$ }
\begin{cases}
(x^4+yx+zx^2,y,z)\\
(x,y^2+z^k,z)
\end{cases}
\end{equation}
The first one is a codimension 1 monic concatenation of
$(x^4+yx,y)$, and the family is $\mathcal A$-equivalent to
codimension $k$ monic concatenations of $(x^4+y^kx+yx^2,y)$.

If both $f$ and $g$ are $A_2$ singularities we have a codimension 1
binary concatenation

\begin{equation}
\begin{cases}
(x^3+yx,y,z)\\
(x,y,z^3+yz)
\end{cases}
\end{equation}
We should consider two cuspidal edges with some type of contact.
First we study the contact between one of the cuspidal edges and the
limiting tangent plane to the other. From \cite[Example 4.15
ii)]{nosso}, there is only one $\mathcal A$-class for any type of
contact and it has codimension 2, a normal form is
\begin{equation}
\begin{cases}
(x^3+y^lx+zx,y,z)\\
(x,y,z^3+yz)
\end{cases}
\end{equation}
The next type of contact is between the two limiting tangent planes.
Using the complete transversal method we obtain the simple bigerms
of codimensions 3 and 4 respectively
\begin{equation}
\begin{cases}
(x^3+yx,y,z)\\
(x^3+zx+x^2y,y,z)
\end{cases}
\text{ and }
\begin{cases}
(x^3+yx,y,z)\\
(x^3+zx,y,z)
\end{cases}
\end{equation}
Based on the previous example, different types of contact between
the limiting tangent planes yield the same germ.

The multiplicity of a bigerm of an $A_2$ and an $A_3$ singularity
overpasses Nishimura's bound for simplicity.

\subsection{Trigerms}

Due to Nishimura's bound we can only have either 3 folds or 2 folds
and a germ of multiplicity 3.

With 3 folds either the trigerm is a stable triple point ($k=1$ in
any of the families below) or it is an augmentation (again by
\cite[Corollary 3.8]{nosso}). The germ $h$ must be an augmentation
of one of the germs $\{(x^2,y),(x^2+y^l,y),(x,y^2)\}$ since they are
the only trigerms of 3 fold singularities from $\mathbb C^2$ to
$\mathbb C^2$ which admit a 1 parameter stable unfolding. Comparing
with the simple trigerms of 3 immersions in \cite{roberta} a trigerm
with 3 fold singularities is simple if it is equivalent to one of
the following

\begin{equation}
\begin{cases}
(x^2,y,z)\\
(x^2+y+z^k,y,z)\\
(x,y^2,z)
\end{cases}
\text{,}
\begin{cases}
(x^2,y,z)\\
(x^2+y^l+z^2,y,z)\\
(x,y^2,z)
\end{cases}
\text{,}
\end{equation}
\begin{equation}
\begin{cases}
(x^2,y,z)\\
(x^2+yz+z^k,y,z)\\
(x,y^2,z)
\end{cases}
\text{and }
\begin{cases}
(x^2,y,z)\\
(x^2+y^2+z^3,y,z)\\
(x,y^2,z)
\end{cases}
\end{equation}
The last case corresponds to Example \ref{exaugconc} v). Notice that
the second family is also a simultaneous augmentation and
concatenation of a codimension 1 germ.

If we have two fold singularities and a cuspidal edge we must
consider two cases. Firstly, the two $A_1$ singularities must be an
augmentation so we study what kind of augmentations together with a
cuspidal edge give simple germs. We use \cite[Theorem 4.12]{nosso}
about the codimension of a cuspidal concatenation. The only simple
germs here are those in Example \ref{exaugcusp} ii) of codimension
$l+1$

\begin{equation}
\begin{cases}
(x,y,z^2)\\
(x,y,z^2+y^2+x^l)\\
(x^3+yx,y,z)
\end{cases}
\end{equation}
with $l\geq 1$. For $l=1$ we get a codimension 2 germ which can be
seen as a monic concatenation.

Secondly, a fold and a cuspidal edge are also an augmentation, so
together with another fold, $h$ might be a simultaneous augmentation
and concatenation. In this case, if the augmentation comes from a
codimension 1 germ, then $h$ is simple so we get the simple trigerms
of the family

\begin{equation}
\begin{cases}
(x,y,z^2)\\
(x,y^2+z^k,z)\\
(x^3+yx,y,z)
\end{cases}
\end{equation}
for $k\geq 1$, which come from the only codimension 1 germ from
$\mathbb C^2$ to $\mathbb C^2$ with a fold and a cusp. For $k=1$ we
get a codimension 1 monic concatenation. If we consider the
codimension 2 germ $f=\{(x^3+yx,y),(x^2,y)\}$, since the germ
$(x,y,z^2)$ is transverse to the strata of any augmentation of $f$,
the simultaneous augmentation and concatenation of $f$ will yield
non-simple germs.

If $h$ were not a simultaneous augmentation and concatenation, the
cuspidal edge with the other fold would be an augmentation too. So
we would have normal forms
\begin{equation}
\begin{cases}
(x,y^2+z^l,z)\\
(x^2+z^k,y,z)\\
(x^3+yx,y,z)
\end{cases}
\end{equation}
However, in the adjacency of these germs there is the germ
$\{(x,y,z^2),(x^2+z^k,y,z),(x^3+yx,y,z)\}$ which is not simple due
to the previous example, so, in this case, $h$ is not simple.

Example 4.24 iii) shows a $3_{\mu}A_1^2$ case which is not simple,
however, the first fold is not the best possible with respect to the
first branch, so we must consider
$\{(x^3+(y^2+z^l)x,y,z),(x,y,z^2),(x,y,z^2+y)\}$.

\subsection{Quadrigerms}

Here all branches must be fold singularities. From Example
\ref{exaugconc} iii) and iv), the only simple quadrigerms are

\begin{equation}
\begin{cases}
(x^2,y,z)\\
(x,y^2,z)\\
(x^2+y+z^l,y,z)\\
(x,y,z^2)
\end{cases}
\end{equation}

From Example \ref{ejem0n-2} ii) we know that there are no simple
pentagerms.

The following table includes all simple multigerms from $\mathbb
C^3$ to $\mathbb C^3$.

\begin{tabular}{| c | c | c |} \hline $\mathcal K$-orbit
& Normal form & $\mathcal{A}_e$-cod \\ \hline $A_1A_1$ & $
\{(x,y,z^2);(x,y,z^2+h(x,y))\}$ & $\mu(h)$ \\ \hline $A_1A_2$ &
$\{(x^3+yx,y,z);(x,y^2+z^k,z)\}$ & $k-1$ \\
 &
$\{(x^3+yx,y,z);(x^2+z^k,y,z)
\}$ & $2(k-1)$ \\
\hline $A_1A_3$ & $\{(x^4+yx+zx^2,y,z);(x,y^2+z^k,z)\}$ & $k$ \\
\hline $A_2A_2$ & $\{(x^3+zx,y,z);(x,y,z^3+yz)\}$ & $1$ \\
& $\{(x^3+y^2x+zx,y,z);(x,y,z^3+yz)\}$ & $2$ \\
& $\{(x^3+yx,y,z);(x^3+zx+x^2y,y,z)\}$ & $3$ \\
& $\{(x^3+yx,y,z);(x^3+zx,y,z)\}$ & $4$ \\ \hline $3_{\mu}A_1$ &
$\{(x^3+(y^2+z^{\mu+1})x,y,z);(x,y,z^2)\}$ & $\mu+1$ \\
& $\{(x^3+(y^2+z^{\mu+1})x,y,z);(x,y^2,z)\}$ & $2\mu$ \\
\hline $4_1^kA_1$ & $\{(x^4+yx+z^kx^2,y,z);(x,y,z^2)\}$ & $k$ \\
\hline
 $3_{\mu}A_2$ & $\{(x^3+(y^2+z^{\mu+1})x,y,z);(x,y,z^3+yz)\}$ & $\mu+2$ \\ \hline \hline
 $A_1A_1A_1$ & $\{(x^2,y,z);(x^2+y+z^k,y,z);(x,y^2,z)\}$ & $k-1$ \\
 & $\{(x^2,y,z);(x^2+y^k+z^2,y,z);(x,y^2,z)\}$ & $k$ \\
 & $\{(x^2,y,z);(x^2+yz+z^k,y,z);(x,y^2,z)\}, k\geq 2$ & $k$ \\
 & $\{(x^2,y,z);(x^2+y^2+z^3,y,z);(x,y^2,z)\}$ & $4$ \\ \hline
 $A_1A_1A_2$ & $\{(x,y,z^2);(x,y,z^2+y^2+x^k);(x^3+yx,y,z)\}$ &
 $k+1$ \\
  & $\{(x,y,z^2);(x,y^2+z^k,z);(x^3+yx,y,z)\}$ & $k$ \\ \hline $3_{\mu}A_1A_1$ & $\{(x^3+(y^2+z^{\mu+1})x,y,z);(x,y,z^2);(x,y,z^2+y)\}$ & $\mu+2$ \\ \hline
  \hline $A_1A_1A_1A_1$ & $\{(x^2,y,z);(x,y^2,z);(x^2+y+z^k,y,z);(x,y,z^2)\}$
  & $k$ \\ \hline
\end{tabular} \\ \\
Where $h(x,y)$ is a simple function in two variables, $\mu(h)$
stands for the Milnor number of $h$ and $k\geq 1$ unless stated
otherwise.

 \noindent
R. Oset Sinha,\\
Departament de Geometria i Topologia,\\
Universitat de València,\\
Dr. Moliner 50, 46100, Valencia, Spain.\\
raul.oset@uv.es\\
\\
M. A. S. Ruas, R. Wik Atique,\\
Instituto de Ci\^encias Matem\'aticas e de Computa\c{c}\~ao - USP,\\
Av. Trabalhador s\~ao-carlense, 400 - Centro,\\
CEP: 13566-590 - S\~ao Carlos - SP, Brazil.\\
maasruas@icmc.usp.br, rwik@icmc.usp.br


\begin{thebibliography}{G-M-R-R}

\bibitem {arnold} {\sc{V. I. Arnol'd}} {\it{Simple singularities of curves.}} (Russian. Russian summary) Tr. Mat. Inst. Steklova 226 (1999), Mat. Fiz. Probl. Kvantovoi Teor. Polya, 27--35; translation in
Proc. Steklov Inst. Math. 1999, no. 3 (226), 20–28

\bibitem {brucegaffney} {\sc{J. W. Bruce and T. Gaffney}} {\it{Simple singularities of maps $(\C,0)\rightarrow (\C^2,0)$}} J. London Math. Soc. Ser. (2) 26 (1982) 464–474.

\bibitem {brucekirk} {\sc{J. W. Bruce, N. P. Kirk and A. A. du Plessis}} {\it{Complete transversals and the classification of singularities.}} Nonlinearity 10 (1997), no. 1, 253–275.

\bibitem {cati} {\sc{C. Casonatto, M. C. Romero Fuster and R. Wik Atique}} {\it{Topological invariants of stable maps of oriented 3-manifolds in $\mathbb{R}^4$}}. Preprint.

\bibitem {robertamond} {\sc{T. Cooper, D. Mond and R. Wik Atique}} {\it{Vanishing topology of codimension 1 multi-germs over R and C}}. Compositio Math 131 (2002), no. 2, 121–160.

\bibitem {Goryunov} {\sc{V. Goryunov}} {\it{Singularities of projections of complete
intersections}}. J. Soviet Math. 27 (1984), 2785–2811.

\bibitem {hobbskirk} {\sc{C. A. Hobbs and N. P. Kirk}} {\it{On the classification and bifurcation of multigerms of maps from surfaces to 3-space}}. Math. Scand. 89 (2001), no. 1, 57–96.

\bibitem {houston} {\sc{K. Houston}} {\it{Augmentation of singularities of smooth mappings}}. Internat. J. Math. 15 (2004), no. 2, 111–124.

\bibitem {houstonkirk} {\sc{K. Houston and N. Kirk}} {\it{On the classification and geometry of corank 1 map-germs from three-space to four-space}}. Singularity theory (Liverpool, 1996), xxii, 325–351, London Math. Soc. Lecture Note Ser., 263, Cambridge Univ. Press, Cambridge, 1999.

\bibitem {kolgushkin} {\sc{P. A. Kolgushkin and R. R. Sadykov}} {\it{Simple singularities of multigerms of curves.}} Rev. Mat. Complut. 14 (2001) 311–344.

\bibitem {mather}{\sc{J. N. Mather}}
{\it{Stability of $\mathcal C^{\infty}$ mappings IV: Classification of stable maps by $\R$-algebras}}. Publ. Math. IHES, 37, (1983), 223-248.

\bibitem {marartari} {\sc{W. L. Marar and F. Tari}}. {\it{On the
geometry of simple germs of co-rank $1$ maps from $\R^3$ to
$\R^3$}}. Math. Proc. Cambridge Philos. Soc. 119 (1996), no. 3,
469--481.

\bibitem {mararmond} {\sc{W. L. Marar and D. Mond}}. {\it{Multiple point schemes for co-rank $1$ maps}}. J. London Math. Soc. 39 (1989), 553--567.

\bibitem {mond}{\sc{D. Mond}}
{\it{On the Classification of Germs of Maps From $\R^2$ to
$\R^3$}}. Proc. London Math. Soc. (3), 50, 333-369, (1983).

\bibitem {nishisimple} {\sc{T. Nishimura}}. {\it{$\mathcal A$-simple multigerms and $\mathcal L$-simple multigerms}}.
 Yokohama Math. J. 55 (2010), no. 2, 93–104.

\bibitem {tesis}{\sc{R. Oset Sinha}} {\it{Topological invariants of stable maps from 3-manifolds to three-space}}. PhD Thesis. Universitat de Val\`{e}ncia (2009).

\bibitem {nosso}{\sc{R. Oset Sinha, M. A. S. Ruas and R. Wik Atique}} {\it{Classifying codimension two multigerms}}. To appear in Mathematische
Zeitschrift. Available online DOI: 10.1007/s00209-014-1326-2.

\bibitem {yofarid}{\sc{R. Oset Sinha and F. Tari}} {\it{Flat geometry of cuspidal edges}}. Preprint.

\bibitem{duplessis}{\sc{A. du Plessis}} {\it{On the determinacy of smooth map-germs}}. Invent. Math. 58 (1980), no. 2, 107–160.

\bibitem {rieger} {\sc{J. H. Rieger}}
{\it{Families of Maps From the Plane to the Plane}}. J. London
Math. Soc. (2) 36, 351-369, (1986).

\bibitem {riegerruas} {\sc{J. H. Rieger and M. A. S. Ruas}}
{\it{Classification of $\mathcal A$-simple germs from $\K^n$ to $\K^2$}}.  Compositio Math. 79 (1991), no. 1, 99–108.

\bibitem {riegerruas2} {\sc{J. H. Rieger and M. A. S. Ruas}}
{\it{M-deformations of $\mathcal A$-simple $\Sigma^{n-p+1}$-germs from $\R^n$ to $\R^p$, $n\geq p$}}.  Math. Proc. Cambridge Philos. Soc. 139 (2005), no. 2, 333–349.

\bibitem {rrwa}{\sc{J. H. Rieger, M. A. S. Ruas and R. Wik Atique }}
{\it{ M-deformations of  $\mathcal{A}$-simple germs from $\mathbb{R}^n$ to $\mathbb{R}^{n+1}$}}, Math. Proc. Cambridge Philos. Soc. 144 (2008), 181-195.

\bibitem {wall} {\sc{C. T. C. Wall}} {\it{Finite determinacy of smooth map-germs}}. Bull. London Math. Soc. 13 (1981), 481–539.

\bibitem {roberta} {\sc{R. Wik Atique}} {\it{On the classification of multi-germs of maps from
      $\mathbb {C}^2$ to $\mathbb {C}^3$ under $\mathcal
      {A}$-equivalence}}. in J.W.Bruce and F.Tari(eds.) {\it Real and
Complex Singularities},
      Research Notes in Maths Series, Chapman \& Hall / CRC (2000),
119-133.

\bibitem {zhitomirskii} {\sc{M. Zhitomirskii}}
{\it{Fully simple singularities of plane and space curves.}}
Proc. Lond. Math. Soc. (3) 96 (2008), no. 3, 792–812.

 \end{thebibliography}
\end{document}